\documentclass[12pt]{amsart}
\usepackage{amsmath}
\usepackage{amssymb}
\usepackage{mathtools,nccmath,textcomp}
\usepackage{tikz}
\usepackage[all]{xy}
\usepackage{longtable}
\allowdisplaybreaks[1]
\newtheorem{theorem}{Theorem}
\newtheorem{lemma}[theorem]{Lemma}

\newtheorem{corollary}[theorem]{Corollary}

\theoremstyle{definition}
\newtheorem{definition}[theorem]{Definition}

\theoremstyle{remark}
\newtheorem{remark}[theorem]{Remark}

\numberwithin{equation}{section}

\DeclareMathAlphabet{\matheur}{U}{eur}{m}{n}

\newcommand{\m}{\mathrm{m}}
\newcommand{\C}{\mathbb{C}}

\newcommand{\R}{\mathbb{R}}
\newcommand{\Q}{\mathbb{Q}}
\newcommand{\N}{\mathbb{N}}
\newcommand{\Z}{\mathbb{Z}}

\newcommand{\im}{\mathop{\mathrm{Im}}} 

\newcommand{\di}{\mathrm{d}}
\newcommand{\sign}{\mathrm{sign}}

\renewcommand\d{{\mathrm d}}

\renewcommand{\Im}{\operatorname{Im}}
\renewcommand{\Re}{\operatorname{Re}}

\newmuskip\pFqskip
\mathchardef\pFcomma=\mathcode`, 

\newcommand*\pFq[5]{%
  \begingroup
  \begingroup\lccode`~=`,
    \lowercase{\endgroup\def~}{\pFcomma\mkern\pFqskip}%
  \mathcode`,=\string"8000
  {}_{#1}F_{#2}\biggl(\genfrac..{0pt}{}{#3}{#4} \,\,\bigg| \,\, #5\biggr)%
  \endgroup
}
\renewcommand{\Im}{\operatorname{Im}}
\renewcommand{\Re}{\operatorname{Re}}
\renewcommand{\d}{\mathrm d}

\begin{document}
\title[Mahler measures of a family of non-tempered polynomials]{Mahler measures of a family of non-tempered polynomials and Boyd's conjectures}

\author{Yotsanan Meemark}
\address{Department of Mathematics and Computer Science, Faculty of Science, Chulalongkorn University, Bangkok, Thailand 10330} \email{yzm101@yahoo.com}

\author{Detchat Samart}
\address{Department of Mathematics, Faculty of Science, Burapha University, Chonburi, Thailand 20131} \email{petesamart@gmail.com}



\date{\today}

\maketitle
\begin{center}
\textit{In memory of John Tate}
\end{center}
\begin{abstract}
We prove an identity relating Mahler measures of a certain family of non-tempered polynomials to those of tempered polynomials. Evaluations of Mahler measures of some polynomials in the first family are also given in terms of special values of $L$-functions and logarithms. Finally, we prove Boyd's conjectures for conductor $30$ elliptic curves using our new identity, Brunault-Mellit-Zudilin's formula and additional functional identities for Mahler measures.
\end{abstract}

\section{Introduction}\label{S:intro}
The (logarithmic) Mahler measure of a nonzero $d$-variate Laurent polynomial $P(x_1,\ldots,x_d)\in \C[x_1^{\pm 1},\ldots,x_d^{\pm 1}]$ is defined by 
\begin{equation*}
\m(P)=\int_0^1 \cdots \int_0^1 \log |P(e^{2\pi i \theta_1},\ldots,e^{2\pi i \theta_d})| \d\theta_1 \cdots \d\theta_d.
\end{equation*} 
Before stating the problems to be investigated in this article, let us recall the definition of a tempered polynomial in two variables \cite[Sect. III]{RV}. 
\begin{definition}\label{D:tempered}
Let $\displaystyle P=\sum_{(m,n)\in \Z^2}c_{(m,n)}x^m y^n\in \C[x^{\pm 1},y^{\pm 1}]$ and let $\Delta(P)$ be the Newton polygon of $P$. For each side $\tau$ of $\Delta(P)$, we denote the lattice points on $\tau$ (enumerated clockwise) by $\tau(k),k=0,1,2,\ldots $. Then we associate to $\tau$ the univariate polynomial 
\begin{equation*}
P_{\tau}(t) = \sum_{k=0}^\infty c_{\tau(k)}t^k.
\end{equation*}
(Note that this is a finite sum since only a finite number of $c_{\tau(k)}$ are non-zero.) The polynomial $P$ is said to be \textit{tempered} if, for every side $\tau$ of $\Delta(P)$, the zeroes of $P_\tau(t)$ are roots of unity. Otherwise, $P$ is \textit{non-tempered}.
\end{definition}
In this paper, we mainly study Mahler measures of the following $2$-parametric family of Laurent polynomials:
\begin{equation}\label{E:Pac}
P_{a,c}(x,y)=a\left(x+\frac 1x\right)+\left(y+\frac 1y\right)+c.
\end{equation}
It is clear from Definition~\ref{D:tempered} that, for $a\ne 0$, $P_{a,c}(x,y)$ is tempered if and only if $|a|=1$. In his seminal paper \cite{Boyd}, Boyd verified numerically that for many integral values of $k$
\begin{equation}\label{E:E}
\m(P_{1,k})\stackrel{?}=c_k L'(\widetilde{E}_k,0),
\end{equation}
where $c_k\in\Q^\times$, $\widetilde{E}_k$ is the elliptic curve corresponding to $P_{1,k}=0$ and $A \stackrel{?}= B$ means the two quantities are equal to at least $25$ decimal places. The numerical values of $\m(P_{1,k})$ and $L'(\widetilde{E}_k,0)$ can be computed rapidly with high precision using standard computer algebra systems such as \texttt{Mathematica} and \texttt{Magma}, yet the identity \eqref{E:E} is notoriously difficult to prove in general. In \cite{LSZ}, Lal\'{i}n, Zudilin, and the second author establish an identity relating $\m(P_{1,k})$, for $0<k<4$, to `half-Mahler' measures of $P_{a,c}$, where $a$ and $c$ are algebraic expressions of $k$ and use it to prove Boyd's conjecture
\begin{equation}\label{E:21}
\m(P_{1,3})=2L'(\widetilde{E}_3,0),
\end{equation}
where $\widetilde{E}_3$ has conductor $21$. The (conjectural) equation \eqref{E:E} is in fact an instance of a more general conjecture, namely the \textit{Bloch-Beilinson conjecture}, which predicts a deep connection between regulators and $L$-functions associated to algebraic varieties. 
The link between this conjecture and Mahler measures was first observed by Deninger \cite{Deninger} and was examined extensively by Boyd and Rodriguez Villegas \cite{Boyd,RV}. Another example of tempered polynomials is the family
\begin{equation*}
Q_k = (1+x)(1+y)(x+y)-kxy.
\end{equation*}
Boyd \cite[Tab.~2]{Boyd} found that $\m(Q_k)$ appears to satisfy an identity analogous to \eqref{E:E} for many $k\in \mathbb{Z}$. Again, due to the limitedness of the known techniques, only a handful of these identities have been proven rigorously, as shown in Table~\ref{Ta:Q}.


Temperedness of a two-variable polynomial has certain $K$-theoretic interpretation which potentially leads to a conjecture like \eqref{E:E}. In fact, if $P\in \mathbb{Q}[x,y]$ is a tempered polynomial defining an elliptic curve $E$, then $\m(P)$ is expressible in terms of a regulator integral
\begin{equation*}
r(\{x,y\})[\gamma] = \int_\gamma \log |x| d\arg(y) - \log |y|d\arg(x),
\end{equation*}
where $\gamma$ is a path on $E$. (For more details, the reader is referred to \cite{Deninger,RV}.) Non-tempered polynomials, on the other hand, are less well understood. Boyd's experiment showed that Mahler measures of some non-tempered polynomials in two variables are (conjecturally) $\Q$-linear combinations of logarithms and $L$-values. For example, he found 
\begin{equation}\label{E:W}
\m(y^2+kxy-x^3-bx)\stackrel{?}= \frac{1}{4}\log |b|+r_{k,b} L'(E_{k,b},0)
\end{equation}
for some $k,b\in \Z$ and $r_{k,b}\in\Q^\times$ and hypothesized that the logarithmic term arises from the Mahler measure of $P_\tau(t)=t-b$, which corresponds to a side of the Newton polygon of $y^2+kxy-x^3-bx$. The first proven formula in this family (with $k=3$ and $b=1$) was given recently by Lal\'{i}n and Ramamonjisoa \cite{LR}. Examples of proven Mahler measure formulas for non-tempered polynomials involving logarithms and $L$-values can be found in \cite{LM,Giard}. Due to the sparsity of the known results, it is of great importance to gain further examples of non-tempered polynomials whose Mahler measures encode interesting arithmetic information like \eqref{E:W}.

In the last section of \cite{LSZ}, the following identity is stated without proof: for every $a\ge 1$
\begin{equation*}
\frac{3}{2} \m(P_{a,a^2-1})=\m(Q_{a^2-1})+\log a.
\end{equation*}
The primary goal of this paper is to give a proof of an extended version of this statement.
\begin{theorem}\label{T:Main}
If $a\in [1,\infty)\cup \{\sqrt{-r} \mid r\in (0, \infty)\}$, then the following identity holds:
\begin{equation}\label{E:Main}
\frac{3}{2} \m(P_{a,a^2-1})=\m(Q_{a^2-1})+\log |a|.
\end{equation}
\end{theorem}
This result, which is proven in Section~\ref{S:functional}, gives a direct connection between Mahler measures of a non-tempered family and those of a tempered family. Note also that $Q_{a^2-1}=0$ and $P_{a,a^2-1}=0$ generically define the same elliptic curve (up to isomorphism), which can be written in a Weierstrass form as
\begin{equation*}
E_a : Y^2 = X\left(X^2+\frac{(a^4-6a^2-3)}{4}X+a^2\right).
\end{equation*}
Using Theorem \ref{T:Main} and the proven identities in Table~\ref{Ta:Q}, we immediately obtain some new formulas  for $\m(P_{a,c})$ which are analogous to \eqref{E:W}.

\begin{table}[ht]\label{Ta:Q}
\centering \def\arraystretch{1.1}
    \begin{tabular}{ | c | c | c | c |}
    \hline
    $a$ & Conductor of $E_a$ & $\m(Q_{a^2-1})/L'(E_a,0)$ & Proven by\\ \hline
    $\sqrt{3},\sqrt{-3}$ & $36$ & $1/2, 2$ & Rodriguez Villegas \cite{RV}  \\ \hline
    $\sqrt{2},\sqrt{8},\sqrt{-7}$ & $14$ & $1,6,10$ & Mellit \cite{Mellit} \\ \hline
    $\sqrt{5},\sqrt{-1}$ & $20$ & $2,3$ &  Rogers-Zudilin \cite{RZ} \\ \hline
    \end{tabular}
\caption{Proven formulas for $\m(Q_{a^2-1})$}
\label{T2}
\end{table}

\begin{corollary}\label{Co:Pac}
The following evaluations are true:
\begin{align*}
\m(P_{\sqrt{8},7})&=4 L'(E_{\sqrt{8}},0)+\log 2 ,\\
\m(P_{\sqrt{5},4})&=\frac 43 L'(E_{\sqrt{5}},0)+\frac 13\log 5 ,\\
\m(P_{\sqrt{3},2})&=\frac 13 L'(E_{\sqrt{3}},0)+\frac 13\log 3 ,\\
\m(P_{\sqrt{2},1})&=\frac 23 L'(E_{\sqrt{2}},0)+\frac 13\log 2 ,\\
\m(P_{\sqrt{-1},-2})&=2 L'(E_{\sqrt{-1}},0),\\
\m(P_{\sqrt{-3},-4})&=\frac 43 L'(E_{\sqrt{-3}},0)+\frac 13\log 3 ,\\
\m(P_{\sqrt{-7},-8})&=\frac{20}{3} L'(E_{\sqrt{-7}},0)+\frac 13\log 7.
\end{align*}
\end{corollary}

Many recent results relating Mahler measures of two-variate polynomials to elliptic curve $L$-values are accomplished using elegant formulas of Brunault, Mellit and Zudilin \cite{Zudilin,Brunault}. However, these formulas are applicable only for finitely many elliptic  curves over $\Q$, which admit a modular unit parametrization \cite{Brunault2}. Indeed, the motivation for studying the family $P_{a,c}$ is the existence of modular units parametrizing $P_{\sqrt{7},3}=0$, which is equivalent to a classical result of Ramanujan. This eventually leads to a proof of \eqref{E:21}. In Section~\ref{S:munit} and Section~\ref{S:regulator}, we use a similar approach to tackle Boyd's conjectures for elliptic curves of conductor $30$ \cite[Tab.~2]{Boyd}, including
\begin{align}
  \m(Q_3) &= L'(E_2,0), \label{E:mQ3} \\
  \m(Q_9) &= 3L'(E_{\sqrt{10}},0), \label{E:mQ9}\\
  \m(Q_{24}) &= 5L'(E_5,0).  \label{E:mQ24}
\end{align}
Note that there is only one isogeny class of conductor $30$ elliptic curves, so the $L$-values on the right-hand sides of \eqref{E:mQ3}, \eqref{E:mQ9} and \eqref{E:mQ24} all coincide.
To prove these formulas, we first use a newly-discovered modular unit parametrization for $P_{2,3}=0$ and Brunault-Mellit-Zudilin's formula to establish a formula for $\m(P_{2,3})$.
\begin{theorem}\label{T:P23}
The following formula holds:
\begin{equation}\label{E:P23}
\m(P_{2,3}) = \frac{2}{3}(L'(E_2,0)+\log 2).
\end{equation}
\end{theorem}
It is clear that \eqref{E:mQ3} follows from \eqref{E:Main} and \eqref{E:P23}. The remaining two formulas can be deduced by comparing the Mahler measures in the level of regulator via isogenies or isomorphisms between elliptic curves. The details of the proof will be given in Section~\ref{S:regulator}. In the end, we are able to settle all conjectural identities of Boyd in conductor $30$.
\begin{theorem}\label{T:Boyd}
The formulas \eqref{E:mQ3}, \eqref{E:mQ9} and \eqref{E:mQ24} are true.
\end{theorem}
It is worth mentioning that Rogers and Yuttanan made some observations about the conductor $30$ conjectures in \cite{RY} by expressing the $L$-value in terms of certain lattice sums. Using Theorem~\ref{T:Boyd}, we are able to confirm all of their related conjectures. For example, we have 
\begin{align*}
\m(Q_3)=\frac{15}{2\pi^2}\left(F\left(2,\frac53\right)-\frac14F(2,15)\right),
\end{align*}
where 
\begin{align*}
F(b,c)&= (1+b+c+bc)^2 \\
&\quad \times\sum_{n_i=-\infty}^\infty \frac{(-1)^{n_1+n_2+n_3+n_4}}{((6n_1+1)^2+b(6n_2+1)^2+c(6n_3+1)^2+bc(6n_4+1)^2)^2}.
\end{align*}

\section{Proof of Theorem~\ref{T:Main}}\label{S:functional}

We follow an approach of Bertin and Zudilin \cite{BZ} in proving Theorem~\ref{T:Main}. The crucial idea is to compare the derivatives of $\m(P_{a,a^2-1})$ and $\m(Q_{a^2-1})$ with respect to the real parameter $a$ (or $r$ if $a=\sqrt{-r}$). These quantities turn out to be expressible in terms of elliptic integrals, which can be manipulated quite easily by changing variables. We divide the proof into two parts, depending on whether $a$ is real or purely imaginary.

For the sake of brevity, denote
\begin{align*}
G(a)&= \m(Q_{a^2-1}),\\
H(a)&= \m(P_{a,a^2-1}).
\end{align*}
We will also make use of the following standard notations for complete elliptic integrals:
\begin{align*}
K(z)&=\int_0^1\frac{\d x}{\sqrt{(1-x^2)(1-z^2x^2)}}, \qquad
E(z)=\int_0^1\frac{\sqrt{1-z^2x^2}}{\sqrt{1-x^2}}\d x,
\\
\Pi(n,z)&=\int_0^1\frac{\d x}{(1-nx^2)\sqrt{(1-x^2)(1-z^2x^2)}}.
\end{align*}

\subsection{The real cases}
\begin{lemma}\label{L:dh}
For $a\in (1,3)\cup (3,\infty)$, we have
\begin{equation*}
\dfrac{\di}{\di a}H(a)=\frac{2}{\pi}\Re\left(\frac{a^2+3}{a\sqrt{(a-1)^3(a+3)}}K(m)+\frac{(a+1)(a-3)}{a\sqrt{(a-1)^3(a+3)}}\Pi(n,m)\right),
\end{equation*}
where $m=\sqrt{\frac{16a}{(a-1)^3(a+3)}}$ and $n=\frac{4a}{(a-1)(a+3)}$.
\end{lemma}
\begin{proof}
We denote
\begin{align*}
B(x)&:=B_a(x)=a\left(x+\frac{1}{x}\right)+a^2-1,\\ 
\Delta(x)&:=\Delta_a(x)=B(x)^2-4.
\end{align*}
By the quadratic formula, we have
\[y P_{a,a^2-1}(x,y)=y^2+B(x)y+1= (y-y_1(x))(y-y_2(x)),\]
where $y_1(x),y_2(x)=\frac{-B(x)\pm \sqrt{\Delta(x)}}{2}.$ If $\Delta(x)<0$, then $y_1(x)$ and $y_2(x)$ have modulus $1$ and are complex conjugates of each other, since $y_1(x)y_2(x)=1$. Otherwise, we let 
\[y_1(x)=\frac{-B(x)-\sign(B(x))\sqrt{\Delta(x)}}{2},\]
so that $|y_2(x)|<1<|y_1(x)|.$

By Jensen's formula and the fact that $y_1(x)=y_1(x^{-1})$, we obtain
\begin{align*}
H(a)&= \m(P_{a,a^2-1})\\
&=\frac{1}{(2\pi i)^2}\iint_{|x|=|y|=1} \log |P_{a,a^2-1}|\dfrac{\di x}{x}\dfrac{\di y}{y}\\
&=\frac{1}{2\pi i}\int_{|x|=1}\log |y_1(x)| \dfrac{\di x}{x}\\
&=\frac{1}{\pi}\Re\int_0^\pi\log\left(\frac{B\left(e^{i\theta}\right)+\sign\left(B\left(e^{i\theta}\right)\right)\sqrt{\Delta\left(e^{i\theta}\right)}}{2}\right)\di \theta.
\end{align*}
Note that $B\left(e^{i\theta}\right)=2a\cos \theta +a^2-1 > -2$, since $a> 1$. If $-2< B\left(e^{i\theta}\right)<0$, then $\Delta\left(e^{i\theta}\right)<0$, so using 
\[\log\left(\frac{B-\sqrt{\Delta}}{2}\right)+ \log\left(\frac{B+\sqrt{\Delta}}{2}\right)=\log 1=0\]
and $\Re(\log z) =\Re(\log \bar{z})$, one sees that $\Re\left(\log\left(\frac{B\left(e^{i\theta}\right)-\sqrt{\Delta\left(e^{i\theta}\right)}}{2}\right)\right) =0$. Therefore,
\[H(a)=\frac{1}{\pi}\Re\int_0^\pi\log\left(\frac{B\left(e^{i\theta}\right)+\sqrt{\Delta\left(e^{i\theta}\right)}}{2}\right)\di \theta.\]
Simple calculations yield
\begin{equation*}
\dfrac{\di }{\di a} \log \frac{B+\sqrt{\Delta}}{2}=\dfrac{\di}{\di B} \log \left(\frac{B+\sqrt{B^2-4}}{2}\right)\dfrac{\di B}{\di a}=\frac{1}{\sqrt{\Delta}}\left(x+\frac{1}{x}+2a\right).
\end{equation*}
It follows that
\[\dfrac{\di}{\di a}H(a) = \frac{1}{\pi}\Re \int_0^\pi \frac{2(\cos \theta+a)}{\sqrt{\Delta\left(e^{i\theta}\right)}} \di \theta.\]
Then we use the substitution $t=\cos\theta$ to obtain 
\begin{equation}\label{E:dha}
\dfrac{\di}{\di a}H(a) =\frac{2}{\pi}\Re\int_{-1}^1 \frac{t+a}{\sqrt{(2at+a^2-3)(2at+a^2+1)}}\dfrac{\di t}{\sqrt{1-t^2}}.
\end{equation}
Note that the above integral converges for $a\in (1,3)\cup (3,\infty)$. Using the change of variable 
\[t= \frac{2(a^2-3)x^2-(a-1)(a+3)}{-4ax^2+(a-1)(a+3)},\]
one can check easily that 
\begin{align}
t+a &= \frac{a^2+3}{2a}+\frac{(a+1)(a-3)}{2a}\frac{1}{1-nx^2},\label{E:ta}\\
\dfrac{\di t}{\sqrt{(1-t^2)(2at+a^2-3)(2at+a^2+1)}}&= \frac{2}{\sqrt{(a-1)^3(a+3)}}\dfrac{\di x}{\sqrt{(1-x^2)(1-m^2x^2)}},\label{E:dt}
\end{align}
where $m=\sqrt{\frac{16a}{(a-1)^3(a+3)}}$ and $n=\frac{4a}{(a-1)(a+3)}$.
Substituting \eqref{E:ta} and \eqref{E:dt} into \eqref{E:dha} proves the lemma.
\end{proof}

\begin{lemma}\label{L:dg}
For $a\in (1,3)\cup (3,\infty)$, we have
\begin{equation}\label{E:dg}
\dfrac{\di}{\di a}G(a) = \frac{4a}{\pi \sqrt{(a-1)^3(a+3)}}\Re\left( K(m)\right),
\end{equation}
where $m=\sqrt{\frac{16a}{(a-1)^3(a+3)}}$.
\end{lemma}
\begin{proof}
It is proven in \cite[Sect.~3]{BZ} that the derivative of $G(a)$ can be written in terms of a hypergeometric function. In particular, we have 
\[\dfrac{\di}{\di a}G(a) = \frac{2a}{a^2+3}\pFq{2}{1}{\frac{1}{3},\frac{2}{3}}{1}{\frac{27(a^2-1)^2}{(a^2+3)^3}},\]
for any $a\in (1,3)\cup (3,\infty)$. Assume first that $a>3$.
Using the substitution $a=(p+2)/p$, we have $0<p<1$, so we can apply \cite[p.~112, Thm.~5.6]{Berndt} to deduce
\begin{align*}
\dfrac{\di}{\di a}G(a) &= \frac{p(2+p)}{2(1+p+p^2)}\pFq{2}{1}{\frac{1}{3},\frac{2}{3}}{1}{\frac{27p^2(1+p)^2}{4(1+p+p^2)^3}}\\
&= \frac{p(2+p)}{2\sqrt{1+2p}}\pFq{2}{1}{\frac{1}{2},\frac{1}{2}}{1}{\frac{p^3(2+p)}{1+2p}}\\
&= \frac{2a}{\sqrt{(a-1)^3(a+3)}}\pFq{2}{1}{\frac{1}{2},\frac{1}{2}}{1}{\frac{16a}{(a-1)^3(a+3)}}\\
&= \frac{4a}{\pi \sqrt{(a-1)^3(a+3)}}K(m),
\end{align*}
where the last equality follows from the standard hypergeometric representation of $K(m)$.\\
The remaining cases can be verified in a similar manner using the substitution $a=2p+1$ and the identity \cite[Eq.~19.7.3]{DLMF}
\[\Re(K(u))=\frac{1}{u}\Re\left(K\left(\frac{1}{u}\right)\right),\]
which is valid for $u\in (0,\infty)$.
\end{proof}
We shall compare the derivatives of $G(a)$ and $H(a)$ using the preceding lemmas.

\begin{lemma} \label{L:com}
For $a\in (1,3)\cup (3,\infty)$, the equality
\begin{equation*}
\frac32\dfrac{\di}{\di a}H(a)=\dfrac{\di}{\di a}G(a)+\frac{1}{a}
\end{equation*}
holds. 
\end{lemma}
\begin{proof}
By Lemma~\ref{L:dh}, Lemma~\ref{L:dg}, and some rearrangement, it suffices to prove that, for every $a\in (1,3)\cup (3,\infty)$,
\begin{equation}\label{E:KP}
\Pi(n,m)-\frac{a+3}{3(a+1)}K(m)=\frac{\pi}{3}\frac{\sqrt{(a-1)^3(a+3)}}{(a+1)(a-3)},
\end{equation}
where $m$ and $n$ are as given in Lemma~\ref{L:dh}. Surprisingly, this is equivalent to a known result due to Jia \cite{Jia}, which arises from certain problems in particle physics. His proof is based on the observation that, after applying the change of variable $a=\frac{3x-1}{x+1},$ the function on the left-hand side of \eqref{E:KP}, which he called $y(x)$, satisfies the first-order differential equation
\[y' = \frac{1+2x+3x^2}{x(x-1)(1+3x)}y.\]
This can be derived using the derivative formulas for $K(z)$ and $\Pi(n,z)$ (see \cite[Chapter~19]{DLMF} or \cite[p.~13]{LSZ}) and the chain rule. The solution to the above ODE defined on $(-\infty,-1)\cup (1,\infty)$ is 
\[y(x)=C\frac{\sqrt{(x-1)^3(1+3x)}}{x},\]
where $C$ is a constant. Using some initial conditions, he found that $C=-\frac{\pi}{12}$ and the proof is completed.
\end{proof}

\begin{proof}[Proof of Theorem~\ref{T:Main} for $a\in[1,\infty)$]
The case $a=1$ is trivial since $G(1)=H(1)=0$. Suppose $a\in (1,3)\cup (3,\infty)$. Then we have from Lemma~\ref{L:com} that
\begin{align*}
\frac32 H(a) = \frac32\int_1^a \dfrac{\di}{\di u}H(u) \di u = \int_1^a \left(\dfrac{\di}{\di u}G(u)+\frac{1}{u}\right) \di u=G(a)+\log a.
\end{align*}
Since $G(a),H(a),$ and $\log a$ are continuous functions on $(1,\infty)$, it follows that \eqref{E:Main} is valid for all $a\in [1,\infty)$.
\end{proof}
\begin{remark}
The discontinuity at $a=3$ in Lemma~\ref{L:com} is natural due to the fact that the curve $E_3$ is singular. In particular, the Mahler measures $\m(P_{3,8})$ and $\m(Q_{8})$ are not related to elliptic curve $L$-values. On the other hand, Boyd \cite[Tab.~2]{Boyd} conjectured that 
\[ \m(Q_8)\stackrel{?}=5L'(\chi_{-3},-1),\]
where $\chi_{-3}=\left(\frac{-3}{\cdot}\right).$
\end{remark}

\subsection{The complex cases}

We first simplify the problem a bit to avoid working with complex variables. Let $a= \sqrt{-r},$ where $r\in (0,\infty).$ Then we have
\begin{align*}
P_{a,a^2-1}(x,y)\cdot P_{-a,a^2-1}(x,y)&=\left(y+\frac{1}{y}-r-1\right)^2+r\left(x+\frac{1}{x}\right)^2\\
&=\frac{r}{x^2}\left[x^4+\left(\frac{1}{r}\left(y+\frac{1}{y}-r-1\right)^2+2\right)x^2+1\right]. 
\end{align*}
Since $\m(P_{a,a^2-1})=\m(P_{-a,a^2-1})$, it follows that 
\begin{equation*}
2\m(P_{a,a^2-1}) = \m(S_r)+\log r,
\end{equation*}
where $S_r:=S_r(x,y)=y^2+\left(\frac{1}{r}\left(x+\frac{1}{x}-r-1\right)^2+2\right)y+1.$ Here we have applied the transformation $(x^2,y)\mapsto (y,x)$ and \cite[Lem.~7]{Smyth} to modify the Mahler measure slightly on the right-hand side. Consequently, the proof of Theorem~\ref{T:Main} for the complex cases boils down to showing that
\begin{equation}\label{E:Sr}
\m(S_r) = \frac{4}{3}\m(Q_{-r-1})-\frac{1}{3}\log r,
\end{equation}
for real $r>0$.  We will employ the techniques that we use in the real cases to verify \eqref{E:Sr}. 

\begin{lemma}\label{L:Sr}
Let $r=\frac{(1-p)(2+p)}{p(1+p)}$. Then for $p\in (0,1)$, which is mapped bijectively to $r\in(0,\infty)$, we have
\begin{equation*}
\dfrac{\di }{\di r}\m(S_r) = \frac{1}{\pi}\frac{2p(1+p)}{(1-p)\sqrt{1+2p}}\left(K(m)-\frac{(1+p)^2}{2+p}\Pi(n,m)\right),
\end{equation*}
where $m=\sqrt{\frac{p^3(2+p)}{1+2p}}$ and $n=-p^2/(1+2p).$
\end{lemma}

\begin{proof}
Let 
\begin{align*}
B(x)&:=B_r(x)=\frac{1}{r}\left(x+\frac{1}{x}-r-1\right)^2+2,\\
\Delta(x)&:=\Delta_r(x)=B(x)^2-4.
\end{align*}
Since $B(e^{i\theta})\ge 2$, we have $\Delta(e^{i\theta})\ge 0.$ Hence, by the arguments used in the proof of Lemma~\ref{L:dh}, 
\[\m(S_r)=\frac{1}{\pi}\Re\int_0^\pi\log\left(\frac{B\left(e^{i\theta}\right)+\sqrt{\Delta\left(e^{i\theta}\right)}}{2}\right)\di \theta.\]
Then we differentiate both sides with respect to $r$ and do some calculations to obtain
\begin{align*}
\dfrac{\di}{\di r}\m(S_r) &= \frac{1}{\pi r}\Re \int_0^\pi \frac{2\cos\theta+r-1}{\sqrt{4r+(r+1-2\cos\theta)^2}}\di \theta\\
&= \frac{1}{\pi r}\int_{-1}^1 \frac{2t+r-1}{\sqrt{4r+(2t-r-1)^2}}\dfrac{\di t}{\sqrt{1-t^2}}\\
&= \frac{1}{\pi r}\int_{-1}^1 \frac{t+c-1}{\sqrt{(1-t^2)(t^2-2ct+c^2+2c-1)}}\di t,
\end{align*}
where $c=\frac{1}{p(1+p)}$. Note that the polynomial in the denominator has only two real roots, so the process of writing this integral in terms of elliptic integrals is more involved than that in the proof of Lemma~\ref{L:dh}. Let us first rewrite the two quadratic polynomials 
\begin{align*}
S_1(t)&:=1-t^2 = A_1(t-\alpha)^2+B_1(t-\beta)^2,\\
S_2(t)&:=t^2-2ct+c^2+2c-1 = A_2(t-\alpha)^2+B_2(t-\beta)^2,
\end{align*}
where 
\[A_1= -\frac{(1+p)^2}{1+2p},\,B_1=\frac{p^2}{1+2p},A_2=\frac{p(2+p)}{1+2p},\,B_2=\frac{1-p^2}{1+2p},\,\alpha = \frac{p}{1+p},\,\beta = \frac{p+1}{p}.\]
Let $t=\frac{\beta x-\alpha}{x-1}$. Then we have $x=\frac{t-\alpha}{t-\beta}$, whence
\begin{align*}
t+c-1 &= \frac{(\beta+c-1)x-(\alpha+c-1)}{x-1},\\
\dfrac{\di t}{\sqrt{S_1(t)S_2(t)}}&=\dfrac{\di x}{(\alpha-\beta)\sqrt{(A_1x^2+B_1)(A_2x^2+B_2)}}.
\end{align*}
Applying partial fraction decomposition results in 
\begin{align*}
\int_{-1}^1 \frac{t+c-1}{\sqrt{S_1(t)S_2(t)}}\di t &= \int_{\frac{\alpha+1}{\beta+1}}^{\frac{\alpha-1}{\beta-1}}\left(\frac{\beta+c-1}{\alpha-\beta}+\frac{1}{1-x}\right)\dfrac{\di x}{\sqrt{(A_1x^2+B_1)(A_2x^2+B_2)}}\\
&=\int_{\frac{p}{1+p}}^{-\frac{p}{1+p}}\left(-\frac{2+p}{1+2p}+\frac{1}{1-x^2}+\frac{x}{1-x^2}\right)\dfrac{\di x}{\sqrt{(A_1x^2+B_1)(A_2x^2+B_2)}}.
\end{align*}
Using the substitution $u=x^2$, one sees that 
\begin{equation*}
\int_{\frac{p}{1+p}}^{-\frac{p}{1+p}}\frac{x}{1-x^2}\dfrac{\di x}{\sqrt{(A_1x^2+B_1)(A_2x^2+B_2)}}=0.
\end{equation*}
On the other hand, by the substitution $x\mapsto -\sqrt{-\frac{B_1}{A_1}}x=-\frac{p}{1+p}x,$
\begin{align*}
\int_0^{-\frac{p}{1+p}}\dfrac{\di x}{\sqrt{(A_1x^2+B_1)(A_2x^2+B_2)}} &= -\frac{1}{\sqrt{-A_1B_2}}\int_0^1 \dfrac{\di x}{\sqrt{(1-x^2)(1-(A_2B_1/A_1B_2)x^2)}}\\
&=-\frac{1+2p}{(1+p)\sqrt{1-p^2}}K\left(\sqrt{\frac{p^3(2+p)}{(1+p)^2(p^2-1)}}\right).
\end{align*}
Hence we have 
\begin{equation}\label{E:1}
\int_{\frac{p}{1+p}}^{-\frac{p}{1+p}}\dfrac{\di x}{\sqrt{(A_1x^2+B_1)(A_2x^2+B_2)}}=-\frac{2(1+2p)}{(1+p)\sqrt{1-p^2}}K\left(\sqrt{\frac{p^3(2+p)}{(1+p)^2(p^2-1)}}\right).
\end{equation}
By the same transformation, one can easily deduce
\begin{equation}\label{E:2}
\int_{\frac{p}{1+p}}^{-\frac{p}{1+p}}\frac{1}{1-x^2}\dfrac{\di x}{\sqrt{(A_1x^2+B_1)(A_2x^2+B_2)}}=-\frac{2(1+2p)}{(1+p)\sqrt{1-p^2}}\Pi\left(\left(\frac{p}{1+p}\right)^2, \sqrt{\frac{p^3(2+p)}{(1+p)^2(p^2-1)}}\right).
\end{equation}
In the final step, we alter the arguments of the elliptic integrals in \eqref{E:1} and \eqref{E:2} using the following identities \cite[Eq.~15.8.1, 16.16.8]{DLMF}: 
\begin{align*}
K(\sqrt{z}) &= \frac{1}{\sqrt{1-z}}K\left(\sqrt{\frac{z}{z-1}}\right),\\
\Pi(n,\sqrt{z}) &= \frac{1}{(1-n)\sqrt{1-z}}\Pi\left(\frac{n}{n-1},\sqrt{\frac{z}{z-1}}\right),
\end{align*}
which hold whenever $|\arg(1-z)|<\pi$ and $|\arg(1-n)|<\pi$.
\end{proof}

\begin{lemma}\label{L:Qr}
Under the same assumption as those in Lemma~\ref{L:Sr}, we have
\[\dfrac{\di}{\di r}\m(Q_{-r-1})=\frac{p(1+p)}{\pi \sqrt{1+2p}}K(m).\]
\end{lemma}
\begin{proof}
This is again a consequence of \cite[Eq.~(5)]{BZ} and \cite[p.~112, Thm.~5.6]{Berndt}. See the proof of Lemma~\ref{L:dg} for further details.
\end{proof}

We can now give a comparison between the two derivatives in the previous two lemmas.
\begin{lemma}\label{L:com2}
For $r>0$, 
\begin{equation}\label{E:com2}
\dfrac{\di}{\di r} \m(S_r) = \frac{4}{3}\dfrac{\di}{\di r}\m(Q_{-r-1})-\frac{1}{3r}.
\end{equation}
\end{lemma}
\begin{proof}
By Lemma~\ref{L:Sr} and Lemma~\ref{L:Qr}, proving \eqref{E:com2} amounts to verifying that the following identity is true for $0<p<1$:
\begin{equation}\label{E:com3}
\Pi(n,m)-\frac{(2+p)(1+2p)}{3(1+p)^2}K(m) = \frac{\pi}{6}\frac{\sqrt{1+2p}}{(1+p)^2},
\end{equation}
where $m=\sqrt{\frac{p^3(2+p)}{1+2p}}$ and $n=-p^2/(1+2p).$ We shall imitate Jia's arguments, outlined in the proof of Lemma~\ref{L:com}, to prove \eqref{E:com3}. Let us denote the function on the left-hand side of \eqref{E:com3} by $w(p)$. Computing the derivatives of the two elliptic integrals yields
\begin{align*}
\dfrac{\di}{\di p}K(m)&=\frac{3}{p(1-p^2)(2+p)}E(m)-\frac{3(1+p)^2}{p(2+p)(1+2p)}K(m),\\
\dfrac{\di}{\di p}\Pi(n,m) &= \frac{1+2p}{p(1-p^2)(1+p)^2}E(m)-\frac{1}{p(1+p)}K(m)+\frac{1+3p}{(1-p^2)(1+p)^2(1+2p)}\Pi(n,m).
\end{align*}
Hence it is easily seen that $w(p)$ satisfies 
\begin{equation*}
w'=-\frac{1+3p}{(1+p)(1+2p)}w.
\end{equation*}
This differential equation has a general solution on the interval $(0,1)$ of the form 
\[w(p)=C\frac{\sqrt{1+2p}}{(1+p)^2},\]
where $C$ is constant. Letting $p\rightarrow 0^+$, one immediately sees
\[C= \Pi(0,0)-\frac23K(0)= \frac{\pi}{2}-\frac{\pi}{3}=\frac{\pi}{6}\]
and we have \eqref{E:com3}.
\end{proof}

\begin{proof}[Proof of Theorem~\ref{T:Main} for $a=\sqrt{-r}$]
Integrating both sides of \eqref{E:com2} results in
\begin{equation*}
\m(S_r)=\frac{4}{3}\m(Q_{-r-1})-\frac{1}{3}\log r +c,
\end{equation*}
for some constant $c$. Note that as $r\rightarrow\infty$
\begin{align*}
\m(S_r)&=\m\left(\frac{r}{x^2}\left(x^2y+\frac{1}{r}(x^2y^2-2x^3y+4x^2y-2xy+x^2)+\frac{1}{r^2}(y(x^2-x+1)^2)\right)\right)\\
&= \log r +\m(x^2y)+O(1/r)= \log r+O(1/r),
\\
m(Q_{-r-1})&=\m\left(r\left(xy+\frac{1}{r}(xy+(1+x)(1+y)(x+y))\right)\right)\\
&= \log r +\m(xy)+O(1/r) =\log r+O(1/r).
\end{align*}
As a consequence, $\m(S_r)-\frac43\m(Q_{-r-1})+\frac13\log r \rightarrow 0$ as $r\rightarrow \infty$, which implies $c=0.$ This proves \eqref{E:Sr}, so \eqref{E:Main} is valid for all $a\in\{\sqrt{-r}\mid r\in(0,\infty)\}.$
\end{proof}

\section{Modular unit parametrization in level $30$}\label{S:munit}
The aim of this section is to prove Theorem~\ref{T:P23}; i.e.,
\begin{equation*}
\m(P_{2,3}) = \frac{2}{3}(L'(E_2,0)+\log 2).
\end{equation*}
Following the notation in \cite{Zudilin}, for each $a\in\N$, we define a level $30$ modular unit $g_a(\tau)$ as
\begin{equation*}
g_a(\tau):= q^{15B_2(a/30)}\prod_{\substack{n \geq 1 \\ n \equiv a \bmod 30}}(1-q^n)\prod_{\substack{n \geq 1 \\ n \equiv -a \bmod 30}}(1-q^n), \qquad q= e^{2\pi i \tau},
\end{equation*}
where $B_2(x)= x^2-x+1/6$. Let $\tilde{x}(\tau),x_0(\tau),$ and $\tilde{y}(\tau)$ be functions on the upper half-plane $\mathcal{H}=\{z\in \mathbb{C} \mid \Im z>0\}$ defined by 
\begin{align*}
\tilde{x}(\tau)&=2\frac{\eta(2\tau)\eta(6\tau)\eta(10\tau)\eta(30\tau)}{\eta(\tau)\eta(3\tau)\eta(5\tau)\eta(15\tau)}\\
&=\frac{2}{g_1(\tau)g_3^2(\tau)g_5^2(\tau)g_7(\tau)g_9^2(\tau)g_{11}(\tau)g_{13}(\tau)g_{15}^2(\tau)},\\
x_0(\tau)&= \frac{\tilde{x}(\tau)}{2}= (g_1(\tau)g_3^2(\tau)g_5^2(\tau)g_7(\tau)g_9^2(\tau)g_{11}(\tau)g_{13}(\tau)g_{15}^2(\tau))^{-1},\\
\tilde{y}(\tau)&= -\left(\frac{\eta(\tau)\eta(5\tau)\eta(6\tau)\eta(30\tau)}{\eta(2\tau)\eta(3\tau)\eta(10\tau)\eta(15\tau)}\right)^2\\
&= -g_1^2(\tau)g_5^4(\tau)g_7^2(\tau)g_{11}^2(\tau)g_{13}^2(\tau),
\end{align*} 
where $\eta(\tau)$ is the Dedekind eta function. In the following theorem, we will show that the curve $P_{2,3}(x,y)=0$ can be parametrized by $\tilde{x}(\tau)$ and $\tilde{y}(\tau)$. 
\begin{theorem}\label{T:ME} The following identity holds:
\begin{equation}\label{E:MP}
2\left(\tilde{x}(\tau)+\frac{1}{\tilde{x}(\tau)}\right)+\tilde{y}(\tau)+\frac{1}{\tilde{y}(\tau)}+3=0.
\end{equation}
\end{theorem}
\begin{proof}
Let $G(\tau)=\eta(\tau)\eta(3\tau)\eta(5\tau)\eta(15\tau)$. Then it can be shown using Newman's theorem and Ligozat's theorem \cite[Thm.~1.64 and ~1.65]{Ono} that \[G(\tau),G(\tau)\tilde{x}(\tau),G(\tau)\tilde{x}(\tau)^{-1},G(\tau)\tilde{y}(\tau),G(\tau)\tilde{y}(\tau)^{-1}\in M_2(\Gamma_0(30)),\] where $M_k(\Gamma_0(N))$ denotes the space of weight $k$ modular forms for $\Gamma_0(N).$ Hence we have 
\[F(\tau):=G(\tau)\left(2\left(\tilde{x}(\tau)+\frac{1}{\tilde{x}(\tau)}\right)+\tilde{y}(\tau)+\frac{1}{\tilde{y}(\tau)}+3\right)\in M_2(\Gamma_0(30)).\]
By Sturm's theorem \cite{Sturm}, a weight $2$ modular form for $\Gamma_0(N)$ is identically zero if its first $[\Gamma_0(1):\Gamma_0(N)]/6$ Fourier coefficients vanish. Since 
\[ [\Gamma_0(1):\Gamma_0(30)]=30\left(1+\frac12\right)\left(1+\frac13\right)\left(1+\frac15\right)=72, \]
we only need to show that the first twelve coefficients of $F(\tau)$ are zero. These can be easily computed using a computer.
\end{proof}
\begin{remark} To the best of our knowledge, the modular equation \eqref{E:MP} never appeared in the literature. We discovered it via a modular parametrization of the conductor $30$ elliptic curve  
\[E_2: Y^2= X\left(X^2-\frac{11}{4}X+4\right)\]
and the transformations \cite[Eq.~(11)]{LSZ}
\begin{equation}\label{E:xy}
x=\frac{3X-2Y}{X(X-4)}, \quad y=\frac{3X+2Y}{2X(X-1)}
\end{equation}
with the aid of \texttt{CoCalc} \cite{Sage} and \texttt{qseries} package in \texttt{Maple} \cite{Garvan}. Note, however, that the corresponding $X$ and $Y$ are not modular unit parametrization of $E_2$. It would be interesting to find one for $E_2$ or the curve $(1+x)(1+y)(x+y)-3xy=0$.
\end{remark}
Consider the following CM points
\[\tau_1 = \frac{-15+\sqrt{-15}}{60},\quad \tau_2= \frac{-5+\sqrt{-5}}{30},\quad \tau_3= \frac{5+\sqrt{-5}}{30},\quad \tau_4= \frac{15+\sqrt{-15}}{60}.\] 
We shall denote by $[\tau,\tau']$ the geodesic (a circular arc or a vertical straight line orthogonal to the $x$-axis) connecting $\tau$ and $\tau'$. The geodesics $[\tau_1,\tau_2],[\tau_2,\tau_3],$ and $[\tau_3,\tau_4]$ are illustrated in Figure~\ref{F:paths} below. 
\begin{figure}[h!]
\centering
  \includegraphics[width=6 in]{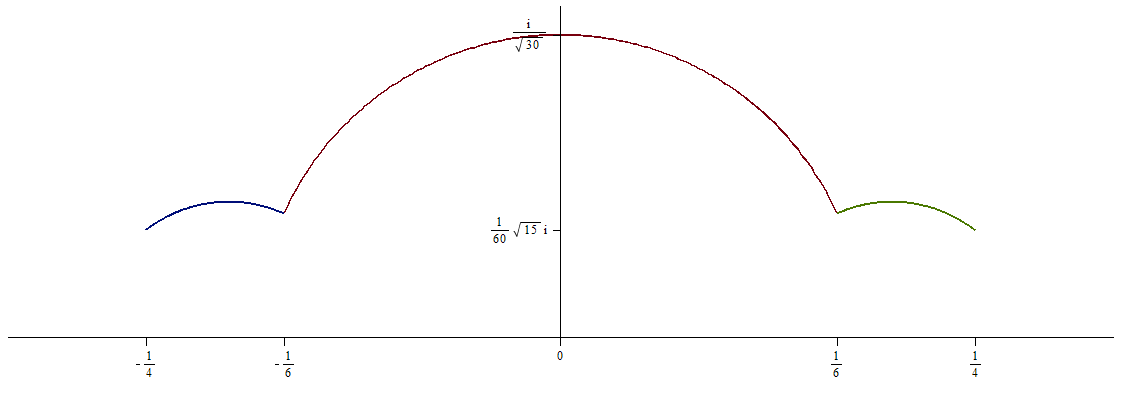}
  \caption{The geodesics $[\tau_1,\tau_2],[\tau_2,\tau_3],[\tau_3,\tau_4]$}
  \label{F:paths}
\end{figure}

Let us briefly explain motivation behind the construction of geodesics in Figure~\ref{F:paths}. It will be shown in the proof of Lemma~\ref{L:cusp} that $\m(P_{2,3})$ can be written as an integral over a path joining two CM points on $E_2$. This integral can then be transformed into an integral of functions involving $\tilde{x}(\tau)$ and $\tilde{y}(\tau)$ from $\tau_1$ to $\tau_4$ on $X_0(30)$, so the integration path becomes a geodesic in $\mathcal{H}$. The strategy is to split $[\tau_1,\tau_4]$ into smaller paths on which the value of $\tilde{y}(\tau)$ is real in order to make the calculations simple. Based on some numerical computation, we find it suitable to divide $[\tau_1,\tau_4]$ into three geodesics as shown above. 

The (normalized) Atkin-Lehner involutions 
\[W_6= 
\frac{1}{\sqrt{6}}\begin{pmatrix}
6 & 1\\
30 & 6
\end{pmatrix}, \quad W_{10}=
\frac{1}{\sqrt{10}}\begin{pmatrix}
10 & 3\\
30 & 10
\end{pmatrix},\quad W_{30}= 
\frac{1}{\sqrt{30}}\begin{pmatrix}
0 & -1\\
30 & 0
\end{pmatrix},\]
of $\Gamma_0(30)$ will be invoked frequently in the proofs of the subsequent lemmas, so we record them here for convenience.
\begin{lemma}\label{L:paths}
For $\tau\in [\tau_i,\tau_{i+1}], i=1,2,3,$ we have $|\tilde{x}(\tau)|=1$ and $\tilde{y}(\tau)\in \mathbb{R}.$
\end{lemma}
\begin{proof}
Let $\tau \in [\tau_1,\tau_2]$. Then $\tau$ can be written in the form 
\[\tau= x+\sqrt{\frac{1}{150}-\left(x+\frac{1}{5}\right)^2}i, \qquad -1/4\le x\le -1/6,\]
so $\tau\overline{\tau}=|\tau|^2= -\frac{12x+1}{30}.$ 
Hence it can be checked easily that
\[W_6\tau = \frac{6\tau+1}{30\tau+6}=-\overline{\tau}.
\]
It follows that $W_6$ sends $(\tilde{x}(\tau),\tilde{y}(\tau))$ to their complex conjugates $(\overline{\tilde{x}(\tau)},\overline{\tilde{y}(\tau)}).$ On the other hand, by \cite[Cor.~2.2]{CL}, we have that $W_6$ acts on the two modular functions as follows: $\tilde{x}(\tau)\mapsto 1/\tilde{x}(\tau),\,\, \tilde{y}(\tau)\mapsto \tilde{y}(\tau)$. Therefore, we have $1/\tilde{x}(\tau)=\overline{\tilde{x}(\tau)}$ and $\tilde{y}(\tau)=\overline{\tilde{y}(\tau)}$, implying 
$|\tilde{x}(\tau)|=1$ and $\tilde{y}(\tau)\in \mathbb{R}$. 
By symmetry, it is not hard to see that the same result holds for $\tau \in [\tau_3,\tau_4].$

For $\tau\in [\tau_2,\tau_3]$, we have $|\tau|^2=\frac{1}{30}$, whence 
\[W_{30}\tau = -\frac{1}{30\tau} = -\overline{\tau}.\]
Thus $W_{30}$ sends $(\tilde{x}(\tau),\tilde{y}(\tau))$ to $(\overline{\tilde{x}(\tau)},\overline{\tilde{y}(\tau)})$. Since the action of $W_{30}$ on $\tilde{x}(\tau)$ and $\tilde{y}(\tau)$ is the same as that of $W_6$ \cite[Cor.~2.2]{CL}, we have the same conclusion as above. 
\end{proof}

\begin{lemma}\label{L:eva}
The following evaluations are true:
\begin{equation*}
\tilde{x}(\tau_1)=\frac{-1-\sqrt{-15}}{4},\quad \tilde{x}(\tau_4)=\frac{-1+\sqrt{-15}}{4},\quad \tilde{y}(\tau_1)=\tilde{y}(\tau_4)= -1.
\end{equation*}
\end{lemma}
\begin{proof}
Employing the action of $W_{10}$, we have 
\[\overline{\tilde{y}(\tau_1)}= \tilde{y}(-\overline{\tau}_1) =\tilde{y}(\tau_4)=\tilde{y}(W_{10}\tau_1)=\frac{1}{\tilde{y}(\tau_1)}.\]
Hence $\tilde{y}(\tau_1)$ has modulus $1$. It is clear from Lemma~\ref{L:paths} and the definition of $\tilde{y}(\tau)$ that $\tilde{y}(\tau_1)=-1.$ Then we apply Theorem~\ref{T:ME} to deduce that $\tilde{x}(\tau_1)$ is a zero of the polynomial $x^2+x/2+1$. It is therefore sufficient to consider the sign of $\Im \tilde{x}(\tau_1)$ numerically in order to get a correct value of $\tilde{x}(\tau_1).$ The values of $\tilde{y}(\tau_4)$ and $\tilde{x}(\tau_4)$ can be obtained using the action of $W_{10}^{-1}$ and the same arguments.
\end{proof}
For meromorphic functions $f$ and $g$ on a smooth curve $C$, we define the real differential form $\eta(f,g)$ on $C$ as
\[\eta(f,g)=\log |f|\di \arg(g)-\log |g| \di \arg(f),\]
where $\di \arg(F) = \Im(\di F/F)$. Hence $\eta(f,g)=0$ if $f$ and $g$ are real. Moreover, it can be shown in a straightforward manner using the definition above that $\eta(f,g)$ is bi-additive; i.e., 
\begin{equation}\label{E:emul}
\eta(f_1 f_2,g)=\eta(f_1,g)+\eta(f_2,g) \text{ and } \eta(f,g_1g_2)=\eta(f,g_1)+\eta(f,g_2).
\end{equation}
\begin{lemma}\label{L:cusp}
The following identity is true:
\begin{equation*}
\m(P_{2,3})=-\frac{1}{2\pi}\int_{9/30}^{i\infty} \eta(x_0(\tau),\tilde{y}(\tau)).
\end{equation*}
\end{lemma}
\begin{proof}
Recall from \cite[Sect.~3]{LSZ} that, for any real numbers $a$ and $c$, $\m(P_{a,c})= \m^{-}(P_{a,c})+ \m^{+}(P_{a,c})$, where $\m^{+}(P_{a,c})$ is an integral of certain function in variable $t$ from $t=-1$ to $t=\max\{-1,-\frac{2+c}{2a}\}$. In addition, we have that if $c/2 < a+1$, then
\[\m^{-}(P_{a,c})=-\frac{1}{2\pi}\int_{[S_{-},S_{+}]}\eta(x,y_{-})=\frac{1}{2\pi}\int_{[S_{-},S_{+}]}\eta(x,y_{+}),\]
where $y_{\pm}(x)=(-B(x)\pm \sqrt{B(x)^2-4})/2,B(x)=a(x+1/x)+c$ and 
\[S_{\pm}=\left(1-\frac{c}{2}\mp \sqrt{\left(1-\frac{c}{2}\right)^2-a^2},\pm\sqrt{\left(1-\frac{c}{2}\right)^2-a^2}\left(1-\frac{c}{2}\mp \sqrt{\left(1-\frac{c}{2}\right)^2-a^2}\right)\right).\]

Substituting $a=2$ and $c=3$ into these expressions, we obtain 
\[\m(P_{2,3})=\m^{-}(P_{2,3})=-\frac{1}{2\pi}\int_{[S_{-},S_{+}]}\eta(x,y_{-})=\frac{1}{2\pi}\int_{[S_{-},S_{+}]}\eta(x,y_{+}),\]
where $y_{\pm}(x)=(-2(x+1/x)-3\pm \sqrt{(2(x+1/x)+3)^2-4})/2$ and $S_{\pm}$ are points on $E_2$ given by 
\[S_{\pm}=(X,Y)= \left(\frac{-1\mp\sqrt{-15}}{2},\frac{15\mp \sqrt{-15}}{4}\right).\]
Here we think of $x$ and $y_+$ as rational functions of $X$ and $Y$ via the transformation \eqref{E:xy} and integrate the differential form $\eta(x,y_+)$ on the curve $E_2$. The points $S_{\pm}$ correspond to 
\[\overline{S}_{\pm}=(x,y)=\left(\frac{-1\pm \sqrt{-15}}{4},-1\right)\]
on the curve $P_{2,3}(x,y)=0$. By Lemma~\ref{L:eva}, we have
\[(\tilde{x}(\tau_1),\tilde{y}(\tau_1))=\overline{S}_{-},\quad (\tilde{x}(\tau_4),\tilde{y}(\tau_4))=\overline{S}_{+}.\]
One sees from Theorem~\ref{T:ME} and symmetry that the function $y_+$ can possibly be corresponding to either $\tilde{y}$ or $\tilde{y}^{-1}$. We need further information to correctly identify $y_+$. On the path $[S_{-},S_{+}]$, we have $|y_+(x)|\le 1$ \cite[Sect.~3]{LSZ} and, by evaluating $\tilde{y}(\tau)$ numerically, $|\tilde{y}(\tau)|\le 1$ on $[\tau_1,\tau_4]$. Therefore, we have
\[
\m(P_{2,3})=\frac{1}{2\pi}\int_{\tau_1}^{\tau_4}\eta(\tilde{x}(\tau),\tilde{y}(\tau))=\frac{1}{2\pi}\int_{\tau_1}^{\tau_4}\eta(x_0(\tau),\tilde{y}(\tau)),
\]
where the second equality is obtained by splitting the path $[\tau_1,\tau_4]$ into $[\tau_1,\tau_2]\cup [\tau_2,\tau_3]\cup [\tau_3,\tau_4]$ and using \eqref{E:emul} and the fact proven in Lemma~\ref{L:paths} that $\tilde{y}(\tau)$ is real on these paths. It should be noted that we are able to decompose the integration path in $\mathcal{H}$ freely  since the singularities of $\tilde{x}(\tau)$ and $\tilde{y}(\tau)$ are confined to the cusps. 

The action of the involution $W_{10}$ results in $\tilde{x}(\tau)\mapsto 1/\tilde{x}(\tau)$ and $\tilde{y}(\tau)\mapsto 1/\tilde{y}(\tau)$ \cite[Cor.~2.2]{CL}. Thus $W_{10}$ sends $x_0(\tau)$ to $1/4x_0(\tau).$ We also have that $W_{10}$ sends $\tau_1$ to $\tau_4$ and $0$ to $3/10$. Consequently, 
\begin{align*}
\int_{\tau_1}^{\tau_4}\eta(x_0(\tau),\tilde{y}(\tau))&= \int_{\tau_1}^{3/10}\eta(x_0(\tau),\tilde{y}(\tau))-\int_{\tau_4}^{3/10}\eta(x_0(\tau),\tilde{y}(\tau))\\
&=\int_{\tau_1}^{3/10}\eta(x_0(\tau),\tilde{y}(\tau))-\int_{W_{10}^{-1}\tau_4}^{W_{10}^{-1}(3/10)}\eta(x_0(W_{10}(\tau)),\tilde{y}(W_{10}(\tau)))\\
&=\int_{\tau_1}^{3/10}\eta(x_0(\tau),\tilde{y}(\tau))-\int_{\tau_1}^{0}\eta(4x_0(\tau),\tilde{y}(\tau)).
\end{align*}
Observe that $\tilde{y}(\tau)$ is real on $[\tau_1,\tau_2], [\tau_2,i/\sqrt{30}],$ and $[i/\sqrt{30},0]$, so by splitting $[\tau_1,0]$ into the union of these three paths one sees that 
\[\int_{\tau_1}^{0}\eta(4x_0(\tau),\tilde{y}(\tau))=\int_{\tau_1}^{0}\eta(x_0(\tau),\tilde{y}(\tau)).\]
Plugging this back into the last expression above, we have 
\begin{align*}
\int_{\tau_1}^{\tau_4}\eta(x_0(\tau),\tilde{y}(\tau))&= \int_{0}^{3/10}\eta(x_0(\tau),\tilde{y}(\tau))\\
&=\int_{0}^{i\infty}\eta(x_0(\tau),\tilde{y}(\tau))-\int_{3/10}^{i\infty}\eta(x_0(\tau),\tilde{y}(\tau))\\
&=-\int_{3/10}^{i\infty}\eta(x_0(\tau),\tilde{y}(\tau)),
\end{align*}
where the last equality follows from the fact that both $x_0(\tau)$ and $\tilde{y}(\tau)$ are real on $[0,i\infty].$
\end{proof}
We are now in a good position to apply the following formula of Brunault and Mellit, whose proof is worked out in detail by Zudilin \cite{Zudilin}, in order to prove Theorem~\ref{T:P23}.
\begin{theorem}[Brunault-Mellit-Zudilin]\label{T:BMZ}
Let $N$ be a positive integer and define
\begin{equation*}
g_a(\tau)= q^{NB_2(a/N)/2}\prod_{\substack{n \geq 1 \\ n \equiv a \bmod N}}(1-q^n)\prod_{\substack{n \geq 1 \\ n \equiv -a \bmod N}}(1-q^n), \qquad q:= e^{2\pi i \tau},
\end{equation*}
where $B_2(x)= \{x\}^2-\{x\}+1/6$. Then for any $a,b,c\in \Z$ such that $N \nmid ac$ and $N\nmid bc$,
\[\int_{c/N}^{i\infty} \eta(g_a,g_b)=\frac{1}{4\pi}L(f(\tau)-f(i\infty),2),\]
where $f(\tau)=f_{a,b;c}(\tau)$ is a weight $2$ modular form given by 
\[f_{a,b;c}=e_{a,bc}e_{b,-ac}-e_{a,-bc}e_{b,ac}\]
and 
\[e_{a,b}(\tau)=\frac{1}{2}\left(\frac{1+\zeta_N^a}{1-\zeta_N^a}+\frac{1+\zeta_N^b}{1-\zeta_N^b}\right)+\sum_{m,n\ge 1}\left(\zeta_N^{am+bn}-\zeta_N^{-(am+bn)}\right)q^{mn}, \quad \zeta_N:= e^{\frac{2\pi i}{N}}.\]
\end{theorem}

\begin{proof}[Proof of Theorem~\ref{T:P23}]
Let $f_{30}(\tau)$ be the normalized newform corresponding to the elliptic curve $E_2$ via the modularity theorem and 
let $E_2(\tau)$ be the normalized weight $2$ Eisenstein series; i.e., 
\begin{align*}
f_{30}(\tau)&=\eta(3\tau)\eta(5\tau)\eta(6\tau)\eta(10\tau)-\eta(\tau)\eta(2\tau)\eta(15\tau)\eta(30\tau)\\
&=q-q^2+q^3+q^4-q^5-q^6-4q^7-\cdots,\\
E_2(\tau)&= 1-24\sum_{k,l>0}kq^{kl}.
\end{align*}
(The reader should be warned not to confuse $E_2(\tau)$ with the curve $E_2$ in the remaining part of this proof.)
Using Lemma~\ref{L:cusp}, Eq.~\eqref{E:emul}, and Theorem~\ref{T:BMZ}, we have 
\begin{equation*}
\m(P_{2,3})= -\frac{1}{2\pi^2}L(f,2),
\end{equation*}
where 
\begin{align*}
f(\tau)&= -15(q-q^2+q^3-q^4+13q^5-q^6-5q^8+q^9-13q^{10}+4q^{11}-q^{12}+O(q^{13}))\\
&=-10 f_{30}(\tau)+\frac{5}{24}(E_2(\tau)-4E_2(2\tau)-3E_2(3\tau)+12E_2(6\tau))\\
&\quad +\frac{175}{24}(E_2(5\tau)-4E_2(10\tau)-3E_2(15\tau)+12E_2(30\tau))-45.
\end{align*}
Note that $f_{30}(\tau)$ and the Eisenstein series in the above expression belong to $M_2(\Gamma_0(30))$. Using Sturm's theorem, we obtain this identification for $f(\tau)$ by comparing the first twelve Fourier coefficients of both sides. Hence $L(f,2)$ is a rational linear combination of an $L$-value of $f_{30}(\tau)$ and those of Eisenstein series. For $s>2$, we have
\[L(E_2(n\tau)-1,s)= -\frac{24}{n^s}\sum_{k,l>0}\frac{1}{k^{s-1}l^s}=-\frac{24}{n^s}\zeta(s)\zeta(s-1).\]
Therefore, if $g(\tau)=f(\tau)+10f_{30}(\tau)$, then 
\[L(g,s)= \rho(s)\zeta(s)\zeta(s-1),\]
where
\[\rho(s)= -5\left(1-\frac{4}{2^s}-\frac{3}{3^s}+\frac{12}{6^s}\right)-175\left(\frac{1}{5^s}-\frac{4}{10^s}-\frac{3}{15^s}+\frac{12}{30^s}\right).\]
The Laurent series of $\rho(s)$ and $\zeta(s-1)$ around $s=2$ are
\[\rho(s)=(-8\log 2)(s-2)+O((s-2)^2), \quad\zeta(s-1)=\frac{1}{s-2}+O(1),\]
implying $L(g,2)=-\frac{4}{3}\pi^2 \log 2$.
In summary, we have
\begin{align*}
\m(P_{2,3})&=\frac{1}{2\pi^2}\left(10L(f_{30},2)-L(g,2)\right)\\
&=\frac{5}{\pi^2}L(f_{30},2)+\frac{2}{3}\log 2.
\end{align*}
Finally, we apply the standard functional equation for $L(f_{30},s)$ to acquire the desired result.
\end{proof}

\section{Elliptic regulator and proof of Theorem~\ref{T:Boyd}}\label{S:regulator}

In this section, we will establish relationships between $\m(Q_3),\m(Q_9),$ and $\m(Q_{24})$ using their connection with the regulator map on the second $K$-group of the corresponding elliptic curve. We start by recalling some notations and definitions which will be used throughout this section. For more details, the reader is referred to \cite{RV}.

Let $E$ be an elliptic curve defined over $\C$ and let $\C(E)$ be its function field. Note that the first cohomology group $H^1(E,\R)$ can be viewed as the dual of the group of homologous cycles $H_1(E,\Z)$. With this identification, the regulator map $r: K_2(E)\otimes \Q \rightarrow H^1(E,\R)$ is defined by
\begin{equation*}
r(\{x,y\})[\gamma]= \int_\gamma \eta(x,y),
\end{equation*}
for $\gamma\in H_1(E,\Z)$  \cite{Beilinson,Bloch}.
If $E$ is defined over $\Q$, then $K_2(E)$ is conjectured to be a finitely generated abelian group of positive rank, though it is difficult to compute explicitly in general. However, under certain conditions, which are satisfied by curves defined by tempered polynomials, we can identify each element of $K_2(E)\otimes \Q$ by that of $K_2(\C(E))\otimes \Q$. On the other hand, Matsumoto's theorem yields
\[K_2(\C(E)) \cong \wedge^2 \C(E)^\times / \langle f \otimes (1-f): f\in \C(E), f\ne 0, 1  \rangle.\] 
Since $E(\C)\cong \C / \Z+\tau\Z$ for some $\tau\in \mathcal{H}$, we can write each point  on $E$ as $\beta= a+b\tau$, with $a,b\in\Z$. The regulator function $R_\tau:E(\C)\rightarrow \C$, originally defined by Bloch \cite{Bloch}, is given by
\[R_\tau(\beta) = -\frac{(\im \tau)^2}{\pi}\sum'_{\substack{m,n\in\Z}}\frac{\sin (2\pi(an-bm))}{(m\tau+n)^2(m\overline{\tau}+n)},\]
where the summation excludes $(m,n)=(0,0)$. Observe that $R_\tau$ can be extended linearly to the set $\Z[E(\C)]$ of divisors on $E$. 

Let $\Z[E(\C)]^{-}$ be the set of equivalence classes of divisors on $E$ under the relation $[-P]\sim-[P]$. The diamond operation $\diamond$ is defined by 
\begin{align*}
\diamond : \Z[E(\C)]\times \Z[E(\C)] &\longrightarrow \Z[E(\C)]^-\\
 \left(\sum m_i(P_i),\sum n_j(Q_j)\right) &\longmapsto \sum m_in_j(P_i-Q_j).
\end{align*}
Following notations in many previous papers, we set
\[g(\alpha)= \m(Q_\alpha)=\m((1+x)(1+y)(x+y)-\alpha xy).\]
Recall from Section~\ref{S:intro} that the curve $Q_{a^2-1}=0$ is birational to
\[E_a : Y^2 = X\left(X^2+\frac{(a^4-6a^2-3)}{4}X+a^2\right). \]
Using Jensen's formula and work of Deninger and Rodriguez Villegas \cite{Deninger,RV}, one can rewrite $g(\alpha)$ as
\[g(\alpha) = \frac{1}{2\pi}r(\{x,y\})[\gamma],\]
where $\gamma$ corresponds to a path on $Q_\alpha(x,y)=0$ subject to the conditions $|x|=1$ and $|y|\ge 1$. Following this approach, Mellit \cite{Mellit} proved that 
\begin{equation}\label{E:gk}
g(\alpha) = \frac{c_\alpha}{2\pi }R_{\tau_\alpha}((x)\diamond (y)), \qquad \text{where } c_\alpha=\begin{cases}
-1 & \text{ if } \alpha>0, \alpha\ne 8,\\
1  & \text{ if } \alpha<0, \alpha\ne -1
\end{cases}
\end{equation}
and $\tau_\alpha\in \mathcal{H}$ corresponds to $E_{\sqrt{\alpha+1}}.$ 
Lal\'{i}n then applied the regulator machinery to establish the following functional identity \cite[Thm.~1.1]{Lalin}: for $p>0$, 
\begin{align}
g\left(-\frac{2p^2}{1+p}\right)+g\left(\frac{4(1+p)}{p^2}\right)&= g\left(\frac{2(1+p)^2}{p}\right)+g\left(-\frac{2}{p(1+p)}\right). \label{E:g2}
\end{align}
The particular cases $p=1/2$ and $p=2$ of \eqref{E:g2} can be written as
\begin{align*}
g\left(-\mfrac13\right)+g(24) &= g(9)+g\left(-\mfrac83\right),\\
g\left(-\mfrac83\right)+g(3) &= g(9)+g\left(-\mfrac13\right),
\end{align*}
implying
\begin{equation}\label{E:gcomb}
g\left(-\mfrac83\right)-g\left(-\mfrac13\right) =g(24)-g(9)=g(9)-g(3).
\end{equation}
Therefore, knowing a relationship between $g\left(-\mfrac83\right)-g\left(-\mfrac13\right)$ and any of $\{g(3),g(9),g(24)\}$ would be adequate to settle the conductor $30$ conjectures. Since there are no known functional identities that give such a relationship directly, we will show that it exists by adapting arguments in the proof of \cite[Thm.~1.1]{Lalin}.
\begin{theorem}\label{T:newg}
The following identity is valid:     
\[g\left(-\mfrac83\right)-g\left(-\mfrac13\right) = \mfrac23 g(9).\]
\end{theorem}
\begin{proof}
For consistency between our notations and those in \cite{Lalin}, we define the elliptic curve $C_\alpha$ by
\[C_\alpha : Y^2 = X(X^2+(\alpha^2-4\alpha-8)X+16(\alpha+1)).\]
We have a birational map from the curve $Q_\alpha(x,y)=0$ to $C_\alpha$ via the change of variables
\begin{align*}
x &= \frac{Y+\alpha X}{2(X-4(\alpha+1))},\\
y &= -\frac{(3\alpha+2)X^2+XY+4(\alpha^2-1)Y+4(\alpha+1)^2(\alpha-4)X+32(\alpha+1)^2}{(X-4(\alpha+1))(Y+(\alpha+2)X-8(\alpha+1))},\\
X &=-\frac{4x(x+y+1)}{y},\\
Y &=-\frac{4x(2x^2+2xy+(\alpha+2)y-(\alpha-2)x-\alpha)}{y}.
\end{align*}
The torsion subgroup of $C_\alpha$ has a cyclic subgroup of order $6$ generated by $P=(4(\alpha+1),4\alpha(\alpha+1))$. If $\alpha_0 = \frac{2(1+p)^2}{p}$, then $C_{\alpha_0}$ also has a point of order $2$ which is independent of $P$, namely $Q=(-4p(p+2),0)$. Lal\'{i}n showed that 
\begin{equation}\label{E:diamond}
(x)\diamond (y) = -6(P)-6(2P).
\end{equation}
Moreover, setting $\beta_1=-\frac{2}{p(1+p)}$ and $\beta_2=-\frac{2p^2}{1+p}$, she used the isomorphisms
\begin{align*}
\varphi_1 : C_{\alpha_0}&\longrightarrow C_{\beta_1}\\
(X,Y) &\longmapsto \left(\frac{X+4p(2+p)}{(p+1)^2},\frac{Y}{(p+1)^3}\right)\\
\varphi_2 : C_{\alpha_0}&\longrightarrow C_{\beta_2}\\
(X,Y) &\longmapsto \left(\frac{p^2X+4(2p+1)}{(p+1)^2},\frac{p^3Y}{(p+1)^3}\right)
\end{align*}
to obtain 
\begin{align*}
(x_{\beta_2}\circ \varphi_2)\diamond (y_{\beta_2}\circ \varphi_2) &= 6(P_{\alpha_0}+Q_{\alpha_0})+6(2P_{\alpha_0})\\
(x_{\beta_1}\circ \varphi_1)\diamond (y_{\beta_1}\circ \varphi_1) &= -6(2P_{\alpha_0}+Q_{\alpha_0})+6(2P_{\alpha_0}).
\end{align*}
For our purposes, we choose $p=1/2$ so that the above isomorphisms become $\varphi_1 : C_{9}\rightarrow C_{-8/3}$ and $\varphi_2 : C_{9}\rightarrow C_{-1/3}$ and for convenience we shall write $x_j$ and $y_j$, $j=0,1,2$, for the corresponding functions on $C_9,C_{-8/3}$, and $C_{-1/3}$, respectively. The torsion subgroup of $C_9(\Q)$ is generated by $P=(9,9)$ and $Q=(-9/4,27/8)$ and the base point of $C_9$ at infinity will be denoted by $O$.
We have from the two equations above that
\begin{equation}\label{E:comb}
3((x_{2}\circ \varphi_2)\diamond (y_{2}\circ \varphi_2)-(x_{1}\circ \varphi_1)\diamond (y_{1}\circ \varphi_1))=18(P+Q)+18(2P+Q).
\end{equation}
Consider the rational functions 
\[f_1=\frac{4X+9}{4X+Y},\quad 1-f_1 = \frac{Y-9}{4X+Y},\quad f_2= \frac{X-9}{4X+Y}, \quad 1-f_2=\frac{3X+Y+9}{4X+Y}\]
on the curve $C_9$. Their divisors are 
\begin{align*}
(f_1) &= (O)+2(Q)-(2P)-2(2P+Q),\\
(1-f_1) &= (P)+(P+Q)+(4P+Q)-2(2P+Q)-(2P),\\
(f_2) &= (O)+(P)+(5P)-(2P)-2(2P+Q),\\
(1-f_2) &=(4P)+2(4P+Q)-(2P)-2(2P+Q),
\end{align*}
where $2P=(0,0),4P=(0,-9),5P=(9,-81), P+Q=(-3,9), 2P+Q=(-6,24),$ and $4P+Q=(-6,9)$. 
Applying the diamond operation yields
\begin{align*}
(f_1)\diamond (1-f_1) &= 9(2P)+6(2P+Q)-6(P)-6(P+Q),\\
(f_2)\diamond (1-f_2) &= 7(2P)+8(2P+Q)+2(P)+4(P+Q).
\end{align*}
It follows that 
\[(f_1)\diamond (1-f_1) - 3(f_2)\diamond (1-f_2) = -12(2P)-18(2P+Q)-12(P)-18(P+Q).\]
Since $(f_1)\diamond (1-f_1)$ and $(f_2)\diamond (1-f_2)$ are trivial in $K$-theory, we have 
\begin{equation}\label{E:sim}
18(P+Q)+18(2P+Q) \sim -12(P)-12(2P).
\end{equation}
Putting \eqref{E:diamond}, \eqref{E:comb}, and \eqref{E:sim} together, we obtain
\[3((x_{2}\circ \varphi_2)\diamond (y_{2}\circ \varphi_2)-(x_{1}\circ \varphi_1)\diamond (y_{1}\circ \varphi_1))=2(x_0)\diamond (y_0).\]
Finally, we apply the regulator function and use \eqref{E:gk} to translate the above identity into
\[3\left(g\left(-\mfrac13\right)-g\left(-\mfrac83\right)\right)=-2g(9)\]
as desired.
\end{proof}

\begin{proof}[Proof of Theorem~\ref{T:Boyd}]
By \eqref{E:Main}, \eqref{E:P23}, \eqref{E:gcomb}, and Theorem~\ref{T:newg}, we have
\[L'(E_2,0)= g(3)=\frac13 g(9) = \frac15 g(24).\]
Since the curves $E_2, E_{\sqrt{10}},$ and $E_5$ are in the same isogeny class, the proof is complete.
\end{proof}

\section{Concluding remarks}
The $2$-parametric family $P_{a,c}(x,y)$ apparently has some interesting arithmetic properties. If we choose $c$ properly to be an algebraic expression of $a$, their Mahler measures (or half-Mahler measures) turn out to be related to those of the one-parametric tempered families $P_{1,k}(x,y)$ and $Q_k(x,y)$, as one can see from \cite[Thm.~2]{LSZ} and Theorem~\ref{T:Main} in this article. A key feature of these results is that they can be applied to obtain rigorous proofs for some conjectures of Boyd concerning $\m(P_{1,k})$ and $\m(Q_k)$. The proofs of the two theorems mentioned above rely on, after some manipulations, the existence of identities between the complete elliptic integrals of the first and the third kind. It might be possible to recover other functional identities for Mahler measures by tracking backwards from elliptic integral identities similar to \eqref{E:KP} and \eqref{E:com3}. It would also be desirable to prove Boyd's other conjectures using modular equations in different levels.
\vspace{0.5 in}

\textbf{Acknowledgements}
This research is supported financially by the Thailand Research Fund (TRF) grant from the TRF and the Office of the Higher Education Commission, under the contract number MRG6280045. Part of this work was done during the second author's visit to the Vietnam Institute for Advanced Study in Mathematics (VIASM) in the Summer semester of year 2018, partially supported under the institute's research contract no. NC/2018/VNCCCT. The second author is grateful to the VIASM for their support and hospitality. The authors would like to thank Bruce Berndt and Wadim Zudilin for their valuable comments on an earlier version of this paper. Lastly, the authors thank the anonymous referee for insightful remarks which greatly help to improve the exposition of this paper and for a suggestion about possible approach to completing the proof of Boyd's conductor $30$ conjectures.

\bibliographystyle{amsplain}
\bibliography{Mahler}
\end{document}